\documentclass[11pt,letterpaper,reqno,abstract=true]{scrartcl}

\usepackage{amsmath}
\usepackage{amsfonts}
\usepackage{amssymb}
\usepackage{amsthm}
\usepackage{amscd}

\usepackage{tikz}
\usepackage{tikz-cd}
\usetikzlibrary{cd}

\usepackage[colorlinks=true]{hyperref}
\hypersetup{urlcolor=red, citecolor=blue}

\usepackage{latexsym}
\usepackage{mathrsfs}
\usepackage{psfrag}
\usepackage[dvips]{epsfig}
\usepackage{epsfig}
\usepackage[all]{xy}         
\usepackage{wrapfig}

\usepackage[margin=2.5cm]{geometry}

\theoremstyle{plain}                              

\newtheorem{cor}{Corollary}

\newtheorem*{mainthm}{Theorem}               
\newtheorem{lemma}{Lemma}

\theoremstyle{definition}                         

\newtheorem*{remark}{Remark} 

\theoremstyle{remark}                             

\newcommand{\R}{\mathbb{R}}                     

\newcommand{\ms}[1]{\mathscr{#1}}               

\newcommand{\del}{\partial}

\begin{document}

\title{Cross sections to flows via intrinsically harmonic forms}

\author{Slobodan N. Simi\'c}

\maketitle


\begin{abstract}
  We establish a new criterion for the existence of a global cross
  section to a non-singular volume-preserving flow on a compact
  manifold. Namely, if $\Phi$ is a non-singular smooth flow on a
  compact, connected manifold $M$ with a smooth invariant volume form
  $\Omega$, then $\Phi$ admits a global cross section if and only if
  the $(n-1)$-form $i_X \Omega$ is intrinsically harmonic, that is,
  harmonic with respect to some Riemannian metric on $M$.
\end{abstract}





\date{\today}

%

The goal of this note is to prove a simple geometric criterion for the
existence of a global cross section to a volume-preserving
non-singular flow.

The question of existence of a global cross section to a flow is a
fundamental problem in dynamical systems. Much work has been done on
this question; see for instance, \cite{fried+82, plante72,
  schwartz+57, verj+70, ghys89, sns+97, sns+section+16}. We focus on
volume-preserving non-singular flows and prove that the existence of a
global cross section is equivalent to the property of a certain
canonical invariant differential form being intrinsically
harmonic. More precisely, our main result is the following.

\begin{mainthm}
  Let $\Phi$ be a non-singular smooth \footnote{By smooth we always
    mean $C^\infty$.} flow on a smooth, compact, connected manifold
  $M$. Denote the infinitesimal generator of $\Phi$ by $X$ and assume
  that $\Phi$ preserves a smooth volume form $\Omega$. Then $\Phi$
  admits a smooth global cross section if and only if $i_X \Omega$ is
  intrinsically harmonic.
\end{mainthm}

A smooth differential form $\omega$ on $M$ is called
\textsf{intrinsically harmonic} if there exists a smooth Riemannian
metric $g$ on $M$ such that $\omega$ is $g$-harmonic, i.e.,
$\Delta_g \omega = 0$, where $\Delta_g$ denotes the Laplace-Beltrami
operator on differential forms induced by $g$ (cf., \cite{warner+83,
  jost+rgga+2008}). Recall that $\Delta_g = d \delta_g + \delta_g d$,
where $d$ is the exterior differential and
$\delta_g = (-1)^{n(k+1)+1} \star_g \, d \, \star_g$ (on $k$-forms) is
the adjoint of $d$ relative to the $L^2$-inner product defined by $g$
($\star_g$ denotes the Hodge star operator). A smooth form $\omega$ is
$g$-harmonic if and only if $\omega$ is both closed ($d\omega = 0$)
and co-closed ($\delta_g \omega = 0$, i.e., $\star_g \omega$ is
closed).

A closed (compact and without boundary) submanifold $\Sigma$ of $M$ is
called a \textsf{global cross section} to a (clearly non-singular)
flow $\Phi = \{ \phi_t \}$ on $M$ if it intersects every orbit of
$\Phi$ transversely. It is not hard to see that this guarantees that
the orbit of every point $x \in \Sigma$ returns to $\Sigma$, defining
the \textsf{first-return} or \textsf{Poincar\'e map} $P$ of
$\Sigma$. More precisely, for each $x \in \Sigma$ there exists a
unique $\tau(x) > 0$ (called the \textsf{first-return time}) such that
$P(x) = \phi_{\tau(x)}(x)$ is in $\Sigma$, but
$\phi_t(x) \not\in \Sigma$ for every $t \in (0,\tau(x))$.

The first-return map $P : \Sigma \to \Sigma$ is a diffeomorphism. A
global cross section thus allows us to pass from a flow to a
diffeomorphism. To recover the flow from the first-return map one uses
the construction called \textsf{suspension}. Given a smooth closed
manifold $\Sigma$, a diffeomorphism $f : \Sigma \to \Sigma$ and a
smooth positive function (called the \textsf{ceiling} or \textsf{roof
  function}) $\tau : \Sigma \to \R$, we define
\begin{equation}  \label{eq:quotient}
  M_\tau = \Sigma_\tau/\sim_f, \quad \text{where} \qquad
  \Sigma_\tau = \{ (x,t) : x \in \Sigma, 0 \leq t \leq \tau(x) \},
\end{equation}
and $\sim_f$ is the equivalence relation generated by
$(x,\tau(x)) \sim_f (f(x),0)$. The vertical vector field $\del/\del t$
on $\Sigma_\tau$ projects to a smooth vector field $X$ on
$M_\tau$. The flow of $X$ is called the \textsf{special flow}
associated with $\Sigma$, $f$, and $\tau$. If $\tau = 1$, it is
usually called the \textsf{suspension flow} of $f$ (see
\cite{katok+95}).

It is well-known that if a flow $\Phi$ on $M$ has a global cross
section $\Sigma$ with the first-return map $P$ and the first-return
time $\tau$, then the special flow on $M_\tau$ defined by $\Sigma, P$,
and $\tau$ is smoothly orbit equivalent to $\Phi$.

We state the following result of J. Plante, which will be needed in
the proof.

\begin{mainthm}[\cite{plante72}]
  If a $C^1$ flow on a compact manifold $M$ is transverse to the
  kernel of some non-singular continuous closed 1-form $\omega$ on
  $M$, then it admits a smooth global cross section.
\end{mainthm}

Plante showed that each such form $\omega$ can be $C^0$ approximated
by a closed (in the Stokes sense) continuous 1-form $\hat{\omega}$
with \emph{rational} periods. By Hartman's version of the Frobenius
theorem~\cite{hart02}, the kernel of $\hat{\omega}$ is integrable (and
transverse to the flow). Since $\hat{\omega}$ has rational periods,
its integral manifolds are compact, hence each of them is a global
cross section.

\begin{proof}[Proof of the Main Result]
  $(\Leftarrow)$ Assume $i_X \Omega$ is intrinsically harmonic and let
  $g$ be a smooth Riemannian metric such that
  $\Delta_g (i_X \Omega) = 0$. Then $i_X \Omega$ is co-closed; i.e.,
  the 1-form $\omega = \star_g(i_X \Omega)$ is closed. The following
  lemma can be found as an exercise in, e.g.,
  \cite{lee+smooth+13}. The proof is elementary and is omitted.

  \begin{lemma}    \label{lem:star}
    $\star_g(i_X \Omega) = (-1)^{n-1} g(X,\cdot)$, where $n = \dim M$.
  \end{lemma}

  Since $\omega(X) = (-1)^{n-1} g(X,X) \neq 0$, it follows that, $X$
  is transverse to the kernel of the smooth closed non-singular 1-form
  $\omega$. By the result of Plante stated above, the flow has a
  global cross section.  \\

  \noindent $(\Rightarrow)$ Assume now that $\Phi$ admits a global
  cross section $\Sigma$.

\begin{lemma}    \label{lem:reparam}
  There exists a reparametrization $\tilde{\Phi}$ of $\Phi$ whose
  first-return time with respect to $\Sigma$ is constant.
\end{lemma}

\begin{proof}
  Let $\tau$ and $P$ be the first-return time and first-return map of
  $\Sigma$, respectively. Slightly abusing the notation, we denote by
  $\sim_P$ the equivalence relations on both $\Sigma_\tau$ and
  $\Sigma_1$ generated by $(x,\tau(x)) \sim_P (P(x),0)$ (on
  $\Sigma_\tau$) and by $(x,1) \sim_P (P(x),0)$ (on $\Sigma_1$). Let
  $\Sigma_\tau$, $\Sigma_1$, $M_\tau$, and $M_1$ be defined as in
  \eqref{eq:quotient}.

  It is not hard to see that $\Sigma_1$ and $\Sigma_\tau$ are
  diffeomorphic via the map $S : \Sigma_1 \to \Sigma_\tau$ defined by
  \begin{displaymath}
    S(x,t) = (x, t \tau(x)).
  \end{displaymath}
  Furthermore, $M_\tau = \Sigma_\tau/\! \sim_P$ and
  $M_1 = \Sigma_1/\! \sim_P$ are both diffeomorphic to $M$. To
  simplify the notation, we will identify them both with $M$ via these
  diffeomorphisms. Since $S$ maps equivalence classes to equivalence
  classes, we have the following commutative diagram:
  \begin{displaymath}
    \begin{tikzcd}
      \Sigma_1 \arrow[r, "S"] \arrow[d,"\pi_1"]
      & \Sigma_\tau \arrow[d, "\pi_\tau"] \\
      M \arrow[r, "\text{id}"] & M,
    \end{tikzcd}
  \end{displaymath}
  where $\pi_\tau : \Sigma_\tau \to M$ and $\pi_1: \Sigma_1 \to M$ are
  the corresponding quotient maps.
  
  For simplicity we denote the vertical vector fields on both
  $\Sigma_1$ and $\Sigma_\tau$ by $\frac{\del}{\del t}$.

  Observe that $(\pi_{\tau})_\star(\frac{\del}{\del t}) = X$. Let
  $\tilde{X} = (\pi_1)_\star(\frac{\del}{\del t})$. Since $S_\star$
  maps $\frac{\del}{\del t}$ (on $\Sigma_1$) to a scalar multiple of
  $\frac{\del}{\del t}$ (on $\Sigma_\tau$) , it follows that
  $\tilde{X} = u X$ for some smooth positive function $u : M \to
  \R$. Thus $\tilde{X}$ is a reparametrization of $X$ and $\Sigma$ is
  clearly a global cross section for its flow, $\tilde{\Phi}$, with
  first-return time equal to 1. This completes the proof of the Lemma.
\end{proof}

\begin{lemma}   \label{lem:harmonic}
  Let $X$ be a smooth non-singular vector field on $M$ with flow
  $\Phi$. Assume $E$ is a smooth integrable distribution on $M$ such
  that $TM = \R X \oplus E$ and $E$ is invariant under the flow. If
  $\Omega$ is a smooth volume form invariant relative to $\Phi$, then
  $i_X \Omega$ is intrinsically harmonic.
\end{lemma}

\begin{proof}
  Let $g$ be any smooth metric such that: (1) $g(X,X) = (-1)^{n-1}$ at
  every point, and (2) $X$ is orthogonal to $E$ relative to $g$. By
  multiplying the restriction of $g$ on $E$ by a suitable smooth
  function, we can arrange that the Riemannian volume form be
  precisely $\Omega$ (without affecting (1) and (2)). We claim that
  $i_X \Omega$ is $g$-harmonic. Since $\Omega$ is invariant under the
  flow, $i_X \Omega$ is clearly closed. Indeed:
  \begin{displaymath}
    d i_X \Omega = (d i_X + i_X d) \Omega = L_X \Omega = 0,
  \end{displaymath}
  by Cartan's formula. Let us show that
  $\omega = \star_g (i_X \Omega)$ is also closed. By
  Lemma~\ref{lem:star}, we have $\omega = g(X,\cdot)$. Thus
  $\text{Ker}(\omega) = E$ and $\omega(X) = 1$. Since $E$ is
  integrable, the Frobenius theorem yields
  \begin{displaymath}
    \omega \wedge d \omega = 0.
  \end{displaymath}
  To prove that $\omega$ is closed, it is enough to show that
  $d\omega(X,V) = 0$ and $d\omega(V,W) = 0$, for any two smooth local
  sections $V, W$ of $E$. We have:
  \begin{displaymath}
    0 = (\omega \wedge d\omega)(X,V,W) 
    = \omega(X) \: d\omega(V,W) 
    = d \omega(V,W),
  \end{displaymath}
  so $d \omega(V,W) = 0$. Furthermore:
  \begin{align*}
    d\omega(X,V) & = (i_X d\omega)(V) \\
    & = (i_X d \omega + d i_X\omega)(V) \\
    & = L_X \omega(V) \\
    & = 0.
  \end{align*}
  Thus $d\omega = 0$, completing the proof of the lemma.
\end{proof}

Let $\tilde{\Phi} = \{ \tilde{\phi}_t \}$, $\tilde{X}$, and $u$ be as
in Lemma~\ref{lem:reparam}. Note that
$\tilde{\phi}_1(\Sigma) = \Sigma$ and let $\ms{F}$ be the foliation of
$M$ with leaves $\tilde{\phi}_t(\Sigma)$, for $t \in \R$. Let
$E = T\ms{F}$ be the distribution tangent to $\ms{F}$. Since $\Sigma$
and $\Phi$ are smooth, so is $E$. Moreover, $E$ is invariant under the flow.

Set $\tilde{\Omega} = (1/u) \Omega$. It is clear that $\tilde{\Omega}$
is a volume form and that $i_{\tilde{X}} \tilde{\Omega} = i_X
\Omega$. It therefore suffices to show that
$i_{\tilde{X}} \tilde{\Omega}$ is intrinsically harmonic, which
immediately follows from Lemma~\ref{lem:harmonic}. This completes the
proof.
\end{proof}

\begin{cor}
  Let $\Phi$ be a non-singular smooth flow on a smooth, compact,
  connected manifold $M$, with infinitesimal generator $X$. Assume
  $\Phi$ preserves a smooth volume form $\Omega$. If $i_X \Omega$ is
  intrinsically harmonic, then
  $[i_X \Omega] \neq \mathbf{0} \in H^{n-1}_{\text{\emph{de Rham}}}(M)$.
\end{cor}

\begin{proof}
  By the main result, $\Phi$ admits a global cross section
  $\Sigma$. Since $X$ is transverse to $\Sigma$, $i_X \Omega$ is a
  volume form for $\Sigma$, so $\int_\Sigma i_X \Omega \neq 0$. If
  $i_X \Omega$ were exact, Stokes's theorem would imply
  $\int_\Sigma i_X \Omega = 0$.
\end{proof}

\begin{cor}
  If $X$ is the geodesic vector field of a closed Riemannian manifold
  of negative sectional curvature and $\Omega$ denotes the canonical
  invariant volume form, then $i_X \Omega$ is not harmonic with
  respect to any Riemannian metric.
\end{cor}

\begin{proof}
  It is well-known that $X$ does not admit a global cross section. 
\end{proof}

\begin{remark}
  \begin{enumerate}
  \item[(a)] Intrinsically harmonic closed $k$-forms on $n$-manifolds
    were characterized by E. Calabi \cite{calabi+69} and E. Volkov
    \cite{volkov+08} for $k = 1$, and K. Honda \cite{honda+97} for
    $k = n-1$. For a closed \emph{non-vanishing} $(n-1)$-form $\Theta$
    on a smooth manifold $M$, Honda showed that $\Theta$ is
    intrinsically harmonic if and only if it is
    \emph{transitive}. This means that through every point of $M$
    there passes an $(n-1)$-dimensional submanifold $N$ such that the
    restriction of $\Theta$ to $N$ is a volume form for $N$. Note that
    in our setting where $\Theta = i_X \Omega$, this condition
    strongly suggests that $X$ admits a cross section.

  \item[(b)] Observe that by Lemma~\ref{lem:star}, $i_X \Omega$ is
    $g$-harmonic if and only if the 1-form $X^\flat = g(X,\cdot)$ is
    closed. This in turn is equivalent to $v \mapsto \nabla^g_v X$
    being a symmetric linear operator or equivalently, to $\nabla^g X$
    being a symmetric $(1,1)$-tensor (cf., \cite{petersen+06}, \S
    9.2). We therefore have:
  \end{enumerate}
\end{remark}

\begin{cor}
  Let $\Phi, X$, and $\Omega$ satisfy the assumptions of the main
  result. Then $\Phi$ admits a global cross section if and only if
  there exists a smooth Riemannian metric $g$ on $M$ such that
  $\nabla^g X$ is a symmetric $(1,1)$-tensor.
\end{cor}

\paragraph{Acknowledgement.} This work was partially supported by SJSU
Research, Scholarship, and Creative Activity grants.


\bibliographystyle{amsalpha} 

\begin{thebibliography}{War83}

\bibitem[Cal69]{calabi+69}
Eugenio Calabi, \emph{An intrinsic characterization of harmonic one-forms},
  Global {A}nalysis ({P}apers in {H}onor of {K}. {K}odaira), Univ. Tokyo Press,
  Tokyo, 1969, pp.~101--117.

\bibitem[EG89]{ghys89}
\'{E}tienne Ghys, \emph{Codimension one {A}nosov flows and suspensions},
  Lecture {N}otes in {M}athematics, vol. 1331, pp.~59--72, Springer Verlag,
  1989.

\bibitem[Fri82]{fried+82}
David Fried, \emph{The geometry of cross sections to flows}, Topology
  \textbf{21} (1982), 353--371.

\bibitem[Har02]{hart02}
Philip Hartman, \emph{Ordinary {D}ifferential {E}quations}, second ed.,
  Classics in {A}pplied {M}athematics, vol.~38, {SIAM}, 2002.

\bibitem[Hon97]{honda+97}
Ko~Honda, \emph{On harmonic forms for generic metrics}, Ph.D. thesis, Princeton
  University, 1997.

\bibitem[Jos08]{jost+rgga+2008}
J\"urgen Jost, \emph{Riemannian {G}eometry and {G}eometric {A}nalysis}, fifth
  ed., Universitext, Springer, 2008.

\bibitem[KH95]{katok+95}
Anatole Katok and Boris Hasselblatt, \emph{Introduction to the {M}odern
  {T}heory of {D}ynamical {S}ystems}, Encyclopedia of {M}athematics and its
  {A}pplications, vol.~54, Cambridge {U}niversity {P}ress, 1995.

\bibitem[Lee13]{lee+smooth+13}
John~M. Lee, \emph{Introduction to {S}mooth {M}anifolds}, second ed., Grad.
  Text in Math., vol. 218, Springer, New York, 2013.

\bibitem[Pet16]{petersen+06}
Peter Petersen, \emph{Riemannian {G}eometry}, third ed., Grad. Text in Math.,
  vol. 171, Springer, New York, 2016.

\bibitem[Pla72]{plante72}
Joseph Plante, \emph{Anosov flows}, Amer. J. of Math. \textbf{94} (1972),
  729--754.

\bibitem[Sch57]{schwartz+57}
Saul Schwartzman, \emph{Asymptotic cycles}, Annals of {M}ath. \textbf{66}
  (1957), 270--284.

\bibitem[Sim97]{sns+97}
Slobodan~N. Simi\'c, \emph{Codimension one {A}nosov flows and a conjecture of
  {V}erjovsky}, Ergodic Theory Dynam. Systems \textbf{17} (1997), 1211--1231.

\bibitem[Sim16]{sns+section+16}
Slobodan~N. Simi\'{c}, \emph{Global cross sections for {A}nosov flows}, Ergodic
  Theory Dynam. Systems \textbf{36} (2016), no.~8, 2661--2674.

\bibitem[Ver70]{verj+70}
Alberto Verjovsky, \emph{Flows with cross sections}, Proc. Nat. Acad. Sci.
  U.S.A. \textbf{66} (1970), 1154--1156. \MR{0268916}

\bibitem[Vol08]{volkov+08}
Evegeny Volkov, \emph{Characterization of intrinsically harmonic forms}, J.
  Topol. (2008), no.~3, 643--650.

\bibitem[War83]{warner+83}
Frank~W. Warner, \emph{Foundations of {D}ifferentiable {M}anifolds and {L}ie
  {G}roups}, Graduate {T}exts in {M}ath., no.~94, Springer-Verlag, 1983.

\end{thebibliography}

\providecommand{\bysame}{\leavevmode\hbox to3em{\hrulefill}\thinspace}
\providecommand{\MR}{\relax\ifhmode\unskip\space\fi MR }
\providecommand{\MRhref}[2]{%
  \href{http://www.ams.org/mathscinet-getitem?mr=#1}{#2}
}
\providecommand{\href}[2]{#2}

\end{document}